\documentclass[10pt]{amsart}
\usepackage{latexsym, amsmath,amssymb}
\usepackage{graphicx}
  
\usepackage{float}

\usepackage[colorlinks=true]{hyperref}
\hypersetup{urlcolor=blue, citecolor=blue}

\renewcommand{\epsilon}{\varepsilon}
\newcommand{\newsection}[1]
{\subsection{#1}\setcounter{theorem}{0} \setcounter{equation}{0}
\par\noindent}

\newtheorem{theorem}{Theorem}

\newtheorem{lemma}[theorem]{Lemma}
\newtheorem{corr}[theorem]{Corollary}

\newtheorem{proposition}[theorem]{Proposition}
\newtheorem{deff}[theorem]{Definition}

\newcommand{\bth}{\begin{theorem}}
\newcommand{\ble}{\begin{lemma}}
\newcommand{\bcor}{\begin{corr}}

\newcommand{\bdeff}{\begin{deff}}

\newcommand{\bprop}{\begin{proposition}}
\newcommand{\ele}{\end{lemma}}
\newcommand{\ecor}{\end{corr}}
\newcommand{\edeff}{\end{deff}}

\newcommand{\eprop}{\end{proposition}}

\newcommand{\cd}{\, \cdot\, }

\newcommand{\Rn}{{\mathbb R}^n}

\newcommand{\la}{\lambda}

\newcommand{\e}{\varepsilon}

\renewcommand{\Pi}{\varPi}
\renewcommand{\Re}{\rm{Re} \,}
\renewcommand{\Im}{\rm{Im} \,}
\renewcommand{\epsilon}{\varepsilon}
\newcommand{\sgn}{{\text {sgn}}}

\newcommand{\R}{{\mathbb R}}
\newcommand{\Z}{{\mathbb Z}}

\newcommand{\C}{{\mathbb C}}

\newcommand{\Hn}{{\Bbb H}^n}
\newcommand{\1}{{\rm 1\hspace*{-0.4ex}%
\rule{0.1ex}{1.52ex}\hspace*{0.2ex}}}

\begin{document}

\title[$L^p$ resolvent estimates for simply connected manifolds]
{Concerning $L^p$ resolvent estimates for 
simply connected manifolds of constant curvature}
\thanks{The first author was visiting Johns Hopkins University while this research was
carried out, supported by the China Scholarship Council. The second was supported in part by the NSF, DMS- 1361476}

\author{Shanlin Huang}
\address{Department of Mathematics, Huazhong University of Science and Technology,  Wuhan, China}
\author{Christopher D. Sogge}
\address{Department of Mathematics,  Johns Hopkins University,
Baltimore, MD 21218}

\begin{abstract}
We prove families of uniform $(L^r,L^s)$ resolvent estimates for simply connected manifolds of constant curvature  (negative or positive) that imply the earlier ones for Euclidean space of Kenig, Ruiz and the second author \cite{KRS}.  In the case of the sphere we take advantage of the fact that the half-wave group of the natural shifted Laplacian is periodic.  In the case of hyperbolic space, the key ingredient is a natural variant
of the Stein-Tomas restriction theorem.

\end{abstract}

\maketitle

\newsection{Introduction and main results}

In a paper of Kenig, Ruiz and the second author \cite{KRS}, it was shown that
for each $n\ge3$ if $1<r<s<\infty$ are Lebesgue exponents satisfying
\begin{equation}\label{i.1}
n(\tfrac1r-\tfrac1s)=2\quad \text{and } \, \,
\min \, \bigl(|\tfrac1r-\tfrac12|, \, |\tfrac1s-\tfrac12|\bigr)>\tfrac1{2n},
\end{equation}
then there is a uniform constant $C_{r,s}<\infty$ so that
\begin{equation}\label{i.2}
\|u\|_{L^s(\Rn)}\le C\bigl\| (\Delta_{\Rn}+\zeta)u\bigr\|_{L^r(\Rn)},
\quad u\in C^\infty_0(\Rn),
\end{equation}
where $\Delta_{\Rn}=\tfrac{\partial^2}{\partial x_1^2}
+\cdots +\tfrac{\partial^2}{\partial x_n^2}$ denotes the standard Laplacian
on Euclidean space.  The first condition in \eqref{i.1} is dictated by
scaling and the other part of \eqref{i.1} was also shown in \cite{KRS} to be
necessary for \eqref{i.2}.

Euclidean space of course is the unique simply connected manifold of constant
curvature equal to zero.  The purpose of this paper is to prove sharp
theorems for simply connected manifolds of constant curvature either
$+\kappa$ or $-\kappa$, $\kappa>0$, which naturally imply \eqref{i.2}
when $\kappa \searrow 0$.  In the case of positive curvature $+\kappa$, we
recall that the manifold is the sphere endowed with the metric
which in geodesic polar coordinates $r\theta$, $\theta \in S^{n-1}$ about
any point is given by
\begin{equation}\label{i.3}ds^2=dr^2+\Bigl(\frac{\sin \sqrt\kappa r}{\sqrt \kappa}\Bigr)^{2}
|d\theta|^2, 
\quad  0<r<\tfrac{\pi}{\sqrt\kappa}, 
\end{equation}
while in the case of constant negative curvature $-\kappa$, $\kappa>0$,
it is $\Rn$ endowed with the metric which in any geodesic normal coordinate
system takes the form
\begin{equation}\label{i.4}
ds^2=dr^2+\Bigl(\frac{\sinh \sqrt{-\kappa}r}{\sqrt{-\kappa}}\Bigr)^{2}|d\theta|^2, \quad 0<r<\infty,
\end{equation}
(see e.g., \cite{Chavel}).   The volume elements associated
with these metrics then of course are
\begin{equation}\label{i.5}
dV_\kappa  =\Bigl(\frac{\sin \sqrt\kappa r}{\sqrt\kappa}\Bigr)^{n-1}dr d\theta,
\quad 0<r<\tfrac\pi{\sqrt\kappa}, \, \, \kappa>0,
\end{equation}
and
\begin{equation}\label{i.5'}
dV_{-\kappa}  =\Bigl(\frac{\sinh \sqrt\kappa r}{\sqrt\kappa}\Bigr)^{n-1}dr d\theta,
\quad 0<r<\infty, \, \, -\kappa<0,
\end{equation}
respectively.  The Laplacian associated to the metrics in these coordinates
then is simply given by
\begin{equation}\label{i.6}\Delta_\kappa=\partial_r^2+(n-1)\sqrt\kappa \cot (\sqrt\kappa r) \partial_r
+(\sqrt\kappa \csc (\sqrt\kappa r))^2\Delta_{S^{n-1}}, \quad \kappa>0,
\end{equation} 
and
\begin{equation}\label{i.7}
\Delta_{-\kappa}=\partial_r^2+(n-1)\sqrt\kappa \coth(\sqrt\kappa r)\partial_r
+(\sqrt\kappa \, \text{csch} (\sqrt\kappa r)^2\Delta_{S^{n-1}}, \quad -\kappa<0.
\end{equation}
When $\kappa=1$, we in the case of the standard round unit sphere and
write $dV_{S^n}=dV_1$ and $\Delta_{S^n}=\Delta_1$, while for $\kappa=-1$,
we are in standard hyperbolic space and write $dV_{-1}=dV_{\Hn}$ and
$\Delta_{-1}=\Delta_{\Hn}$.

We can now state our main results.

First, for the case of constant positive curvature we have the following

\begin{theorem}\label{posresolve}
Let ${\mathcal R}\subset \C$ be the region given by all $\zeta\in \C$
satisfying
\begin{equation}\label{i.8}
{\mathcal R}=\{\zeta\in \C: \, \Re \zeta \le (\Im \zeta)^2\}.
\end{equation}
Then for every $n\ge 3$ and $1<r<s<\infty$ satisfying \eqref{i.1} there
is a constant $C_{r,s}$ so that
\begin{multline}\label{i.9}
\|u\|_{L^s(S^n,dV_{S^n})}\le C_{r,s}
\bigl\|\bigl((\Delta_{S^n}-(\tfrac{n-1}2)^2)+\zeta\bigr)u
\bigr\|_{L^r(S^n, dV_{S^n})}, \\ u\in C^\infty(S^n), \, \, \zeta \in {\mathcal R}.
\end{multline}
For such exponents $r,s$ and the same constant $C_{r,s}$, we also have
that for any $\kappa>0$
\begin{multline}\label{i.10}
\|u\|_{L^s(S^n,dV_\kappa)}\le C_{r,s}
\bigl\|\bigl((\Delta_{\kappa}-\kappa (\tfrac{n-1}2)^2)+\zeta\bigr)u
\bigr\|_{L^r(S^n, dV_{\kappa})}, \\ u\in C^\infty, \, \, 
\zeta \in {\mathcal R}_\kappa = \kappa {\mathcal R},
\end{multline}
where $\kappa {\mathcal R}=\{\kappa \zeta: \, \zeta\in {\mathcal R}\}$ is
the $\kappa$-dilate of ${\mathcal R}$.
\end{theorem}

For the case of constant negative curvature, we have the following

\begin{theorem}\label{negresolve}
Let $n\ge 3$.  Then for every $r,s$ as in \eqref{i.1} there is a constant
$C_{r,s}$ so that
\begin{equation}\label{i.11}
\|u\|_{L^s(\Hn,dV_{\Hn})}\le C_{r,s}\bigl\|\bigl((\Delta_{\Hn}+(\tfrac{n-1}2)^2)+\zeta\bigr)u\bigr\|_{L^r(\Hn,dV_{\Hn})}, 
\, \, \, u\in C^\infty_0, \, \, |\zeta|\ge 1.
\end{equation}
For such exponents $r,s$ and the same constant $C_{r,s}$ we also have
that for each $-\kappa$, $\kappa>0$,
\begin{equation}\label{i.12}
\|u\|_{L^s(\Rn, dV_{-k})}\le C_{r,s}
\bigl\|\bigl((\Delta_{-\kappa}+\kappa (\tfrac{n-1}2)^2)+\zeta\bigr)u\bigr\|_{L^r(\Rn,dV_{-\kappa})}, 
\, \, \, u\in C^\infty_0, \, \, |\zeta|\ge \kappa.
\end{equation}
\end{theorem}

In the case of $n=3$, as we shall show at the end of \S4, we can uniform obtain bounds of the form \eqref{i.11} or \eqref{i.12} for all $\zeta \in \C$.

Note that, by sending $\kappa$ to zero, it is straightforward to see that either \eqref{i.10} or 
\eqref{i.12} implies \eqref{i.2} for a given pair of exponents
$1<r<s<\infty$ satisfying $n(\tfrac1r-\tfrac1s)=2$. 
(In the case of \eqref{i.10} one uses the fact that ${\mathcal R}_\kappa$ tends
to $\C\backslash \R_+$ as $\kappa$ goes to zero.)
 Therefore, since it was
shown in \cite{KRS}  that the second part of \eqref{i.1} is necessary for the
Euclidean estimate, we conclude that our assumptions on the exponents in
the above theorems dealing with the other constant curvature cases are
also sharp.

It was shown by Bourgain, Shao, Yao and the second author \cite{BSSY} that
when $s=\tfrac{2n}{n-2}$ and $r=\tfrac{2n}{n+2}$ the condition that
$\zeta\in {\mathcal R}$ is sharp in the sense that the inequality
cannot hold when ${\mathcal R}$ is replaced by any region of the form
$\Re \zeta \le a((|\Im \zeta|)(\Im \zeta)^2$, where $a(\la)\to 0$
as $\la \to +\infty$.  Since \eqref{i.9} for any such $r,s$ implies this
special case involving these dual exponents, we conclude that the condition
on ${\mathcal R}$ is sharp.  See Figure~\ref{fig1} below to see the boundary of the
region ${\mathcal R}$ and ${\mathcal R}_\kappa$ when $0<\kappa\ll 1$.  We shall show in \S 2 that the assumption
\eqref{i.1}, as in the Euclidean case, is sharp for $S^n$.  As we shall see,
\eqref{i.10} is a simple consequence of \eqref{i.9}, and as $\kappa\searrow 0$, we can recover \eqref{i.2} from it.  We have stated the results in
Theorem~\ref{posresolve} for the shifted Laplacian on $S^n$ since the spectrum
of $\sqrt{-\Delta_{S^n}+(\tfrac{n-1}2)^2}$ is $\{k+\tfrac{n-1}2\}_{k=0}^\infty$.

\begin{figure}[H]\label{fig1}
\centering
\includegraphics[scale=1.5]{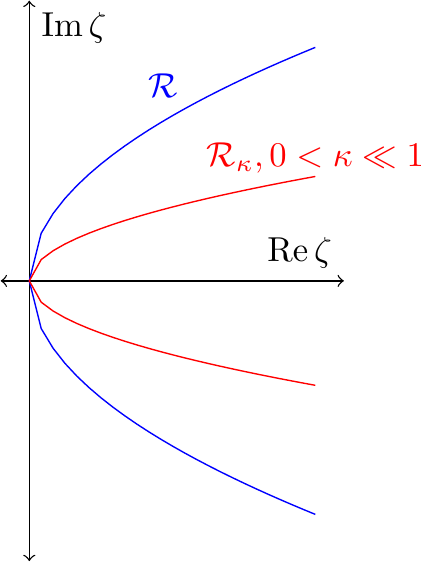}
\caption{Admissible regions for spheres of curvature $\kappa>0$}
\end{figure}

For similar reasons, \eqref{i.11} is stated in terms of the 
shifted Laplacian, $\Delta_{\Hn}+(\tfrac{n-1}2)^2$, due to the fact
that the spectrum of $\sqrt{-\Delta_{\Hn}-(\tfrac{n-1}2)^2}$ is $[0,\infty)$.
Also, as we shall see later, the conditions on $r,s$ in 
Theorem~\ref{negresolve} are necessary, and we can recover \eqref{i.2} from 
\eqref{i.12} by letting $-\kappa \nearrow 0$.

Let us now go over the main ideas in the proof of the theorems.  As in 
\cite{BSSY}, we shall use a formula which relates the resolvent to the solution
of a Cauchy problem for the wave equation involving the shifted Laplacians.
Specifically, if $P^2$ is either the shifted Laplacian $-\Delta_{S^n}+(\tfrac{n-1}2)^2$
or $-\Delta_{\Hn}-(\tfrac{n-1}2)^2$, then we have
\begin{equation}\label{i.14}
\bigl(P^2+(\la+i\mu)^2\bigr)^{-1}
=\frac{\sgn \mu}{i(\la+i\mu)}\int_0^\infty e^{i(\sgn \mu)\la t}
e^{-|\mu|t} \, (\cos tP)\, dt.
\end{equation}
Here one can define $\cos tP$ either via the spectral theorem, or use the fact 
that it gives the unique solution of the Cauchy problem
$$(\partial_t^2-P^2)u(t,x)=0, \, \, u(0,\cdot)=f(\cdot), \, \, \partial_tu(0,\cdot)=0.
$$

In \cite{BSSY} it was shown that favorable restriction type estimates can be used
in conjunction with \eqref{i.14} to prove resolvent estimates.

In the case of the sphere, let $H_k: \, L^2(S^n)\to L^2(S^n)$ denotes projecting onto
spherical harmonics of degree $k=0,1,2,\dots$ (i.e., the projection operator onto the space of eigenfunctions satisfying
$P^2e_k=(k+(\tfrac{n-1}2))^2 \, e_k$.  Then we shall verify the required estimate that
\begin{equation}\label{i.15}
\|H_kf\|_{L^s(S^n)}\le C(1+k)\|f\|_{L^r(S^n)}, \quad
r,s, \, \, \, \text{as in } \, \, \eqref{i.1}.
\end{equation}

In the case of hyperbolic space, we shall show that there
is a constant $C<\infty$ so that one has the uniform bounds
\begin{multline}\label{i.16}
\bigl\| \1_{[\la, \la + \e]}(P)\bigr\|_{L^2(\Hn)\to L^{\frac{2(n+1)}{n-1}}(\Hn)}\le C
\sqrt{\e}\, \la^{\frac{n-1}{2(n+1)}}, 
\quad \la, \e^{-1}\ge 1, \\ 
\text{if } \, \, P=\sqrt{-\Delta_{\Hn}-(\tfrac{n-1}2)^2}.
\end{multline}
Here $\1_{[\la,\la+\e]}$ denotes the indicator function of the interval
$[\la, \la+\e]$, and $\1_{[\la,\la+\e]}(P)$ the function of $P$ defined by this function and the spectral theorem.  If $P$ were the square root of minus the Euclidean
Laplacian and the norms were taken over Euclidean space with Euclidean measure, then the analog of \eqref{i.16} would be equivalent to the Stein-Tomas restriction theorem
for $\Rn$  \cite{Tomas}.  We shall prove \eqref{i.16} by adapting Stein's proof of this result using his complex interpolation scheme.  It seems that the requirement that $\la$ be bounded away from zero is necessary for \eqref{i.16}, unlike its Euclidean counterpart.  Using \eqref{i.16}, following an earlier argument of the second author, 
\cite{Sthesis}, we shall be able to adapt the proof of the Euclidean resolvent
estimates \eqref{i.2} in \cite{KRS} to obtain the hyperbolic variants \eqref{i.11} in 
Theorem~\ref{HST}.

The rest of the paper is organized as follows.  In \S 2 we shall prove Theorem~\ref{posresolve}.  Then \S3 will be devoted to the proof of \eqref{i.16}, and
in \S4 we shall show how they can be used to prove Theorem~\ref{negresolve}, the
uniform $L^r\to L^s$ resolvent estimates for hyperbolic space.  Along the way, we
shall use a few results that are well known to experts and essentially in the
literature.  For the sake of completeness, their proofs will be presented
in an appendix.

\newsection{Uniform resolvent estimates for $S^n$}

In this section we shall give the proof of Theorem~\ref{posresolve}.

First, let us quickly go over some standard facts concerning Fourier analysis
on $S^n$.  Further details can be found in many places, such as in Stein and Weiss~\cite{SW} and \cite{SHang}.

We first recall that we have the orthogonal decomposition
$$L^2(S^n)=\bigoplus^{\infty}_{k=0}{\mathcal H}_k,$$
where ${\mathcal H}_k$ are the spherical harmonics of degree $k=0,1,2,\dots$.  Thus
${\mathcal H}_k$ is the space of restrictions to $S^n$ of harmonic homogeneous polynomials of degree $k$, and so
\begin{equation}\label{2.1}
\bigl(-\Delta_{S^n}+(\tfrac{n-1}2)^2\bigr)e_k(x)=\bigl(k+(\tfrac{n-1}2)\bigr)^2e_k(x),
\quad \text{if } \, e_k\in {\mathcal H}_k,
\end{equation}
and
$$d_k=\dim {\mathcal H}_k=
 \left( \begin{array}{c}
n+k \\
n \end{array} \right)
-
 \left( \begin{array}{c}
n+k -2\\
n \end{array} \right)=\frac{2k^{n-1}}{(n-1)\!}+O(k^{n-2}).
$$
If $\{e_{k,j}\}_{j=1}^{d_k}$ is an orthonormal basis of ${\mathcal H}_k$, then
the projection operator $H_k$ onto this space has kernel
\begin{equation}\label{2.2}
H_k(x,y)=\sum_{k=1}^{d_k}e_{k,j}(x)\overline{e_{k,j}(y)}.
\end{equation}
Since $e_{k,j}$ is the restriction to $S^n$ of a homogeneous polynomial of degree
$k$, we have of course
\begin{equation}\label{2.3}
H_k(x,y^*)=(-1)^kH_k(x,y),
\end{equation}
if
\begin{equation}\label{2.4}
y^*=-y, \quad  y\in S^n\subset {\mathbb R}^{n+1}
\end{equation}
is the antipodal map on $S^n$.

We shall require asymptotics for the kernel $H_k(x,y)$.  The ones we require are
essentially given by the classical Darboux formula for Jacobi polynomials
(see \cite{Sthesis}, \cite{Szego}), but we have not been able to find the exact
results we need in the literature.  Instead, in order to establish them, we shall
use the fact that, since, by \eqref{2.1}, the distinct eigenvalues
\begin{equation}\label{2.5}
\la_k=k+\tfrac{n-1}2, \quad k=0,1,2,\dots,
\end{equation}
of 
\begin{equation}\label{2.6}
P=\sqrt{-\Delta_{S^n}+(\tfrac{n-1}2)^2}
\end{equation}
is one, we have that
\begin{equation}\label{2.7}
H_k=\eta(\la_k-P)=\frac1\pi \int_{-\infty}^\infty \Hat \eta(t)
e^{it\la_k}e^{-itP}\, dt,
\end{equation}
provided that $\eta\in C^\infty_0(\R)$ satisfies
\begin{equation}\label{2.8}
\eta(0)=1, \, \, \, \text{and } \, \eta(\tau)=0, \, \, |\tau|\ge 1/2.
\end{equation}
Since $P$ is positive, we have that $\eta(\la_k+P)=0$, by \eqref{2.8}, if
$k=1,2,\dots$.  Thus, by \eqref{2.7} and Euler's formula, we have
\begin{equation}\label{2.9}
H_k(x,y)=\frac1\pi \int_{-\infty}^\infty \Hat\eta(t)e^{it\la_k}\bigl(\cos tP\bigr)(x,y)
\, dt,
\end{equation}
where $(\cos tP)(x,y)$ denotes the kernel of $\cos tP$.  Note also that, by \eqref{2.5},
we have that
\begin{equation}\label{2.10}
\cos\bigl((t+2\pi)P\bigr)=(-1)^{n-1}\cos tP,
\end{equation}
meaning that $\cos tP$ is $2\pi$-periodic when the dimension $n$ is odd and $4\pi$-periodic when $n$ is even.

As we shall see in the appendix, using the Hadamard parametrix, \eqref{2.3}, \eqref{2.9} and \eqref{2.10} we can easily obtain the following

\begin{proposition}\label{prop2.1}
Let $d_{S^n}(x,y)$ denote geodesic distance on $S^n$.  We then have for $k=0,1,2,\dots$
\begin{equation}\label{2.11}
|H_k(x,y)|\le C(1+k)^{n-1}.
\end{equation}
Moreover, for sufficiently large $k$, we can find functions $a_\pm(k;r)$ so that
\begin{equation}\label{2.12}
H_k(x,y)=\la_k^{\frac{n-1}2}\sum_\pm a_\pm\bigl(k; d_{S^n}(x,y)\bigr)
e^{\pm i\la_kd_{S^n}(x,y)}, \quad\text{if } \, 
\la_k^{-1}\le d_{S^n}(x,y)\le \tfrac{3\pi}4,
\end{equation}
where for every $j=0,1,2,\dots$ 
\begin{equation}\label{2.13}
|\partial_r^ja_\pm(k; r)|\le C_jr^{-j}, \quad\text{if } \, 
\la_k^{-1}\le d_{S^n}(x,y) \le\tfrac{3\pi}4.
\end{equation}
We also have that, for the same functions $a_\pm(k;\, \cdot\, )$,
\begin{multline}\label{2.14}
H_k(x,y)=(-1)^k \la_k^{\frac{n-1}2}\sum_\pm a_\pm\bigl(k; d_{S^n}(x,y^*)\bigr)e^{\pm i d_{S^n}(x,y^*)}, \\ \text{if } \, \tfrac{\pi}4\le d_{S^n}(x,y)\le \pi -\la_k^{-1}.
\end{multline}
\end{proposition}

In order to prove the desired bounds \eqref{i.15} for the harmonic projection operators
and also to be able to prove uniform estimates for a localized version of the
resolvent operators in \eqref{i.14}, we require bounds for certain simi-classical
Fourier integral operators (a.k.a. singular oscillatory integral operators).  To state
them for the generality we shall require throughout, let us assume that $g$ is
a smooth Riemannian metric on $\Rn$ which is close to the Euclidean one.  Assume
further that the injectivity radius of $g$ is ten or more and let
$d_g(x,y)$ be the Riemannian distance function (well defined near the diagonal) and
$B_r(x)=\{y\in \Rn: \, d_g(x,y)<r\}$ if, say, $0<r<10$.  Then, as we shall see
in the appendix, by using Stein's oscillatory integral theorem \cite{St} and H\"ormander's \cite{Hos}, it is not difficult
to obtain the following

\begin{proposition}\label{prop2.2}
Let $g$ be as above and assume that
$$a(x,y)\in C^\infty\bigl(B_2(0)\times B_2(0)\backslash \{(x,x): \, x\in B_2(0)\bigr)$$
satisfies
\begin{equation}\label{2.15}
|a(x,y)|\le C\bigl(d_g(x,y)\bigr)^{2-n}, \quad\text{if } \, \,
d_g(x,y)\le \la^{-1},
\end{equation}
and for every multi-index $\alpha$ we have
\begin{equation}\label{2.16}
|\partial^\alpha_{x,y}a(x,y)|\le C_\alpha \la^{\frac{n-3}2}\bigl(d_g(x,y)\bigr)^{-\frac{n-1}2-|\alpha|}, \quad d_g(x,y)\ge \la^{-1},
\end{equation}
when $x,y\in B_2(0)$. We then have that if $r,s$ are as in \eqref{i.1}
\begin{multline}\label{2.17}
\Bigl\|\, \int e^{i\la d_g(x,y)}a(x,y) f(y)\, dy \, \Bigr\|_{L^s(B_1(0))}
\\
\le C_{r,s}\|f\|_{L^r(B_1(0))}, \quad
f\in C^\infty_0(\Rn), \, \, \, \text{supp }f\subset B_1(0).
\end{multline}
Additionally, if $g$ is sufficiently close to the Euclidean metric in the $C^\infty$
topology, the constant $C_{r,s}$ in \eqref{2.17} depends only on the constant $C$ in
\eqref{2.15} and finitely many
of the constants in \eqref{2.16}.
\end{proposition}

Let us now see how we can use these two results to prove \eqref{i.15}, which says that
\begin{equation}\label{2.18}
\|H_k\|_{L^r(S^n)\to L^s(S^n)}\le C_{r,s}(1+k),
\end{equation}
if $r,s$ are as in \eqref{i.1}.  By compactness, it suffices to show that
\begin{equation}\label{2.19}
\|H_kf\|_{L^s(S^n)}\le C_{r,s}(1+k)\|f\|_{L^r(S^n)}, \quad
\text{if } \, \text{supp } f\subset B_1(x_0),
\end{equation}
with constant $C_{r,s}$ independent of the center $x_0\in S^n$ of the unit-radius ball.

If we choose $\alpha\in C^\infty(\R_+)$ satisfying
\begin{equation}\label{2.20}
\alpha(r)=1, \, \, r\le \delta, \quad \text{and } \, \,
\alpha(r)=0, \, \, r\ge 2\delta,
\end{equation}
and let
$$\tilde H_k(x,y)=\alpha(d_{S^n}(x,y))H_k(x,y),$$
then clearly, if $\delta>0$ is small enough, by the above Propositions, the
integral operator with this kernel $\tilde H_k$ satisfies
\begin{equation}\label{2.21}
\|\tilde H_k f\|_{L^s(S^n)}\le (1+k)\|f\\_{L^r(S^n)}, \quad 
\text{supp } \, f\subset B_1(x_0).
\end{equation}
Similarly, since the map $x\to x^*$ mapping $S^n$ to itself preserves the volume
element, we have that 
\begin{equation}\label{2.22}
\|\tilde H^*_kf\|_{L^s(S^n)}\le C_{r,s}(1+k)\|f\|_{L^r(S^n)}, \quad
\text{supp } f\subset B_1(x_0).
\end{equation}

The remaining piece
$$T_k=H_k-\tilde H_k-\tilde H^*_k,$$
by Proposition~\ref{prop2.1}, is an oscillatory integral operator
of the form
$$T_kh(x)=(1+k)^{\frac{n-1}2}\sum_\pm 
\int_{S^n} \alpha_\pm(k; x,y)e^{\pm i\la_k d_{S^n}(x,y)} \, h(y)\, dV_{S^n}(y),
$$
where
$$\alpha_{\pm}(k;x,y)=0, \quad d_{S^n}(x,y)\notin [\delta,\pi-\delta],$$
and with bounds independent of $k$,
$$|\nabla^\beta_{x,y}\alpha_\pm(k; x,y)|\le C_\beta .$$
Therefore, the desired bounds for it are a consequence of the following
special case of Stein's oscillatory integral theorem \cite{St}.

\begin{lemma}\label{lemma2.3}
Let $(M,g)$ be an $n\ge 2$ dimensional Riemannian manifold.  Assume that the
injectivity radius of $M$ is larger than $R$ and that $M$ is either compact
or of bounded geometry,
and let $d_g(\cd, \, \cd)$ be the associated Riemannian distance function.  Assume further that $\alpha\in C^\infty(M\times M)$
satisfies (in terms of covariant derivatives)
\begin{equation}\label{2.23}
|\nabla_x^{\beta_1}\nabla_y^{\beta_2}\alpha(x,y)|\le C_{\beta_1,\beta_2},
\end{equation}
and
\begin{equation}\label{2.24}
\alpha(x,y)=0, \quad \text{if } \, d_g(x,y)\notin [\delta,R-\delta],
\end{equation}
for some $\delta>0$.  Then if we set
\begin{equation}\label{2.25}
I_\la f(x)=\int e^{i\la d_g(x,y)}\alpha(x,y) \, f(y)\, dV_g(y),
\end{equation}
we have 
\begin{equation}\label{2.26}
\|I_\la f\|_{L^q(M)}\le C\la^{-\frac{n}q}\|f\|_{L^p(M)},
\end{equation}
where $C$ depends on finitely many of the constants in \eqref{2.23} and
\begin{equation}\label{2.27}
1\le p\le 2, \quad q=\tfrac{n+1}{n-1}p', \, \, \tfrac1p+\tfrac1{p'}=1.
\end{equation}
\end{lemma}

Since, by Gauss' lemma the phase function $d_g(x,y)$ of the oscillatory integral
in \eqref{2.26} satisfies the $n\times n$ Carleson-Sj\"olin condition defined
in \cite[\S2.2]{FIO} on the support of $\alpha$ (by \eqref{2.24}), \eqref{2.26}
follows from Corollary~2.2.3 in \cite{FIO}.

To see why this yields our claim that
\begin{equation}\label{2.28}
\|T_kf\|_{L^s(S^n)}\le C(1+k)\|f\|_{L^r(S^n)},
\end{equation}
where $r,s$ are as in \eqref{i.1},  we first note that, by Proposition~\ref{prop2.1},
$$I_{k}=(1+k)^{-\frac{n-1}2}T_k$$
is as in Lemma~\ref{lemma2.3}.  Therefore, if we choose
$$p=\tfrac{2n}{n+1}, \quad q=\tfrac{n+1}{n-1}p'=\tfrac{2n(n+1)}{(n-1)^2},$$
we have, by \eqref{2.26},
\begin{equation}\label{2.29}
\|T_kf\|_{L^{\frac{2n(n+1)}{(n-1)^2}}(S^n)}\le C(1+k)^{\frac{n-1}2}(1+k)^{-\frac{(n-1)^2}{2(n+1)}}=C(1+k)^{\frac{n-1}{n+1}}\|f\|_{L^{\frac{2n}{n+1}}(S^n)}.
\end{equation}

We also have the trivial bounds
\begin{equation}\label{2.30}
\|T_kf\|_{L^\infty(S^n)}\le C(1+k)^{\frac{n-1}2}\|f\|_{L^{\frac{2n}{n+1}}(S^n)}.
\end{equation}
Since 
$$\frac{n-3}{2n}=\frac{(n-1)^2}{2n(n+1)}\cdot \frac{(n+1)(n-3)}{(n-1)^2},$$
$$1-\frac{(n+1)(n-3)}{(n-1)^2}=\frac4{(n-1)^2},$$
and
$$1=\frac{n-1}{n+1}\cdot \frac{(n+1)(n-3)}{(n-1)^2}+\frac{n-1}2\cdot \frac4{(n-1)^2},
$$
if we interpolate between \eqref{2.29} and \eqref{2.30}, we conclude that
\begin{equation}\label{2.31}
\|T_k\|_{L^{\frac{2n}{n-3}}(S^n)}\le C(1+k)\|f\|_{L^{\frac{2n}{n+1}}(S^n)}.
\end{equation}
Since, by the same argument, the adjoint of $T_k$ also enjoys these bounds, we
conclude that we also have that
\begin{equation}\label{2.32}
\|T_kf\|_{L^{\frac{2n}{n-1}}(S^n)}\le C(1+k)\|f\|_{L^{\frac{2n}{n+3}}(S^n)}.
\end{equation}

The estimate \eqref{2.29} corresponds to the point $A$ in Figure~\ref{fig2} below,
and the last two estimates correspond to the points $\alpha$ and $\alpha'$, respectively.  The points $(\tfrac1r, \tfrac1s)$ on the open segment connecting
these two points correspond to the exponents in \eqref{i.1}, and so, by
interpolation, \eqref{2.31} and \eqref{2.32} yield \eqref{2.28}, which completes
the proof of \eqref{2.18}.

\begin{figure}[H]\label{fig2}
\centering
\includegraphics[scale=1.0]{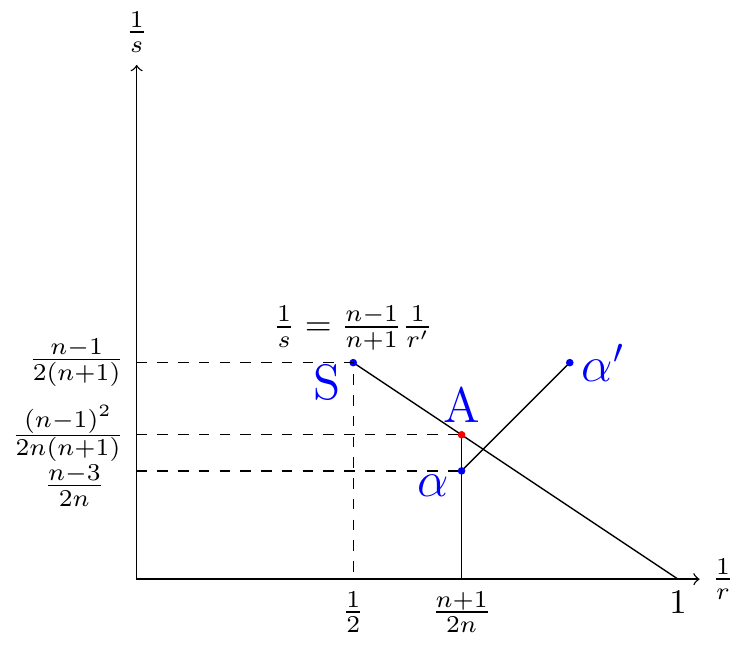}
\caption{Interpolation argument}
\end{figure}

Let us now focus on the proof of Theorem~\ref{posresolve}.  We require the following
result whose proof we postpone until the appendix since it was essentially established
in \cite{DKS}, \cite{BSSY} and \cite{Sef}.  Its proof is a routine stationary
phase calculation.

\begin{proposition}\label{prop2.4}
Fix an even function $0\le \rho \in C^\infty_0(\R)$ satisfying
$$ \rho(t)=1, \, \, |t|\le 1/2, \quad\text{and } \, \, 
 \rho(t)=0, \, \, \, |t|\ge1.$$
Then if $P=\sqrt{-\Delta_{S^n}+(\tfrac{n-1}2)^2}$ and $n\ge 3$ and we set
\begin{equation}\label{2.33}
R_0^{\la,\mu}f=\frac{\text{sgn }\mu}{i(\la+i\mu)}\int_0^\infty  \rho(t) e^{i(\sgn \mu)\la t}
e^{-|\mu|t} \, (\cos tP)\, dt,
\end{equation}
the kernel can be written as
\begin{equation}\label{2.34}
\sum_\pm a_\pm(\la;x,y)e^{\pm i\la d_{S^n}(x,y)} +O\bigl((d_{S^n}(x,y))^{2-n}\bigr),
\end{equation}
where $a_\pm$ vanishes when $d_{S^n}(x,y)$ is close to $\pi$, and moreover,
satisfies \eqref{2.15}-\eqref{2.16} with constants independent of $\la \ge 1$.
\end{proposition}

Using this result, Proposition~\ref{2.2}, Lemma~\ref{lemma2.3} and \eqref{2.19}, we can prove the first part of Theorem~\ref{posresolve}.

Indeed, we first notice that by the first two of these results and the Hardy-Littlewood-Sobolev inequality we have that
\begin{equation}\label{2.35}
\|R_0^{\la,\mu}f\|_{L^s(S^n)}\le C_{r,s}\|f\|_{L^r(S^n)},\end{equation}
for $r,s$ as in \eqref{i.1}.  We also note that if we set
$$m_{\la,\mu}(\tau)=\frac{\text{sgn }\mu}{i(\la+i\mu)}\int_0^\infty  
(1-\rho(t))\, 
 e^{i(\sgn \mu)\la t}
e^{-|\mu|t} \, (\cos t \tau)\, dt,$$
then we clearly have for every $N=1,2,3,\dots$
$$\la \, |m_{\la,\mu}(\tau)|\le C_N(1+|\la-\tau|)^{-N}, \quad \text{if }\, \tau \ge 0,
\, \, \text{and } \, \, \la, \, |\mu|\ge1/2.$$
Therefore, by \eqref{i.14},
$$R^{\la,\mu}_1f=(\Delta+(\tfrac{n-1}2)^2)^{-1}f-R_0^{\la,\mu}f=\sum_{k=0}^\infty m_{\la,\mu}(\la_k)H_k,$$
must, by \eqref{2.19} satisfy for $\la, |\mu|\ge 1/2$,
\begin{align*}
\|R^{\la,\mu}_1f\|_{L^s(S^n)}&\le \sum_{k=0}^\infty |m_{\la,\mu}| \,  \|H_kf\|_{L^s(S^n)}
\\ &\le C\Bigl(\sum_{k=0}^\infty \la^{-1}(1+k)(1+|\la-k|)^{-3}\Bigr)\, \|f\|_{L^r(S^n)}
\\
&\le C'\|f\|_{L^r(S^n)}.
\end{align*}

Based on this estimate we know that we have the uniform bounds in \eqref{i.9} provided that
$$\zeta=(\la+i\mu)^2, \quad \la, \, |\mu|\ge 1/2.$$
Thus we have proven that we have the uniform bounds when $\zeta \in {\mathcal R}$
and $\text{Re }\zeta \ge 1$.  Since the remaining cases follow from Sobolev
estimates, the proof of \eqref{i.9} is complete.

To complete the proof of the theorem (modulo the proofs of the Propositions), we
just need to show that \eqref{i.10} follows from \eqref{i.9} and a simple scaling
argument.  To whit, we claim that if we have
for a certain $r,s$ as in \eqref{i.1} and $\zeta\in \C$
\begin{equation}\label{2.36}\|u\|_{L^s(S^n,dV_{S^n})}\le C_{r,s}
\bigl\|\bigl((\Delta_{S^n}-(\tfrac{n-1}2)^2)+\zeta\bigr)u
\bigr\|_{L^r(S^n, dV_{S^n})}, \quad u\in C^\infty(S^n),\end{equation}
then for the same constant $C_{r,s}$, we must have for a given $\kappa>0$
\begin{equation}\label{2.37}\|u\|_{L^s(S^n,dV_\kappa)}\le C_{r,s}
\bigl\|\bigl((\Delta_{\kappa}-\kappa (\tfrac{n-1}2)^2)+\kappa\zeta\bigr)u
\bigr\|_{L^r(S^n, dV_{\kappa})},  u\in C^\infty.\end{equation}

We recall that $\Delta_{S^n}=\Delta_1$ and that $dV_{S^n}=dV_1$.
To use this we note  that, for, say,
$u_1\in C^\infty_0((0,\pi)\times S^{n-1}$, if we set
$$u_\kappa(\cdot ,\theta)=u_1(\sqrt\kappa \, \cdot  \, ,\theta)\in C^\infty_0((0, \pi/\sqrt\kappa)\times S^{n-1}), $$ then for $0<r<\pi$, we have
\begin{align*}&(\partial^2_r+(n-1)\cot r\,  \partial_r+(\csc r)^2\Delta_{S^{n-1}})u_1(r,\theta)
\\
&=\kappa^{-1} (\partial_r^2u_\kappa)( r/\sqrt\kappa,\theta)+(\sqrt\kappa)^{-1} 
\cot( r)(\partial_ru_\kappa)( r/\sqrt\kappa,\theta)+(\csc r)^2\Delta_{S^{n-1}}u_\kappa( r/\sqrt\kappa,\theta)
\\
&=\kappa^{-1}\Bigl[(\partial_r^2u_\kappa)(t,\theta)+\sqrt\kappa
\cot(\sqrt\kappa t)(\partial_ru_\kappa)(t,\theta)+(\sqrt\kappa \csc \sqrt\kappa t)^2\Delta_{S^{n-1}}u_\kappa(t,\theta)\Bigr],  
\end{align*}
with $t=r/\sqrt\kappa$, $0<t< \pi/\sqrt\kappa$.
Thus, if $u\in C^\infty(S^n)$ and $u_1(r,\theta)$ and $u_\kappa(r,\theta)$ are
its polar coordinates representation about a point in $S^n$ with respect to
the metric of constant curvature $1$ and $\kappa>0$, respectively, we have
$$(\Delta_1+z)u_1(\, \cd\, ,\theta)=\kappa^{-1}(\Delta_\kappa+\kappa z)u_\kappa(\, \cd\, /\sqrt\kappa,\theta), \quad z\in \C.$$
Therefore, by \eqref{i.5},
\begin{align*}
\int_{S^n}&|(\Delta_1-(\tfrac{n-1}2)^2+\zeta)u|^r\, dV_1
\\
&=\int_{S^{n-1}}\int_0^\pi 
|(\Delta_1-(\tfrac{n-1}2)^2+\zeta)u_1(t,\theta)|^r \, (\sin t)^{n-1}dt d\theta 
\\
&=\sqrt{\kappa}\int_{S^{n-1}}\int_0^{\pi/\sqrt\kappa}
\bigl|\kappa^{-1}(\Delta_\kappa-\kappa(\tfrac{n-1}2)^2+\kappa\zeta)u_\kappa(t,\theta)\bigr|^r
\, \bigl(\sin(\sqrt\kappa t)\bigr)^{n-1}\, dt d\theta 
\\
&=(\sqrt{\kappa})^n\int_{S^{n-1}}\int_0^{\pi/\sqrt\kappa}
\bigl|\kappa^{-1}(\Delta_\kappa-\kappa(\tfrac{n-1}2)^2+\kappa\zeta)u_\kappa(t,\theta)\bigr|^r
\, \Bigl(\frac{\sin(\sqrt\kappa t)}{\sqrt\kappa}\Bigr)^{n-1}\, dt d\theta 
\\
&=\kappa^{\frac{n}2}\kappa^{-r}\int_{S^n}|(\Delta_\kappa-\kappa(\tfrac{n-1}2)^2 +\kappa\zeta)u|^r
\, dV_\kappa.
\end{align*}
Similarly,
$$\int_{S^n}|u|^s\, dV_1=\kappa^{\frac{n}2}\int_{S^n}|u|^s\, dV_\kappa.$$
Therefore, if we assume that \eqref{2.36} is valid, we have
\begin{align*}
\|u\|_{L^s(dV_\kappa)}&=\kappa^{-\frac{n}{2s}}\|u\|_{L^s(dV_1)}
\\
&\le C_{r,s}\kappa^{-\frac{n}{2s}}\|(\Delta_{1}+\zeta)u\|_{L^r(dV_1)}
\\
&=C_{r,s}\kappa^{-\frac{n}{2s}}\kappa^{\frac{n}{2r}}\kappa^{-1}
\|(\Delta_\kappa-\kappa(\tfrac{n-1}2)^2+\kappa \zeta)u\|_{L^r(dV_\kappa)},
\end{align*}
which yields \eqref{2.37} as claimed by the first part of our assumption in
\eqref{i.1}.

\newsection{Stein-Tomas estimates  for $\Hn$}

If $P=\sqrt{-\Delta_{\Hn}-(\tfrac{n-1}2)^2}$, then the main result in this section is the
following analogue of the Stein-Tomas restriction theorem for Euclidean space.

\begin{theorem}\label{HST}  There is a uniform constant $C$ so that for all $T$, $\la\ge1$ we have
\begin{equation}\label{s1}
T^{\frac12}\bigl\|\1_{[\la, \la+T^{-1}]}(P)\bigr\|_{L^2(\Hn)\to L^{\frac{2(n+1)}{n-1}}(\Hn)}
\le C\la^{\frac{n-1}{2(n+1)}},
\end{equation}
and
\begin{equation}\label{s2}
T\bigl\|\1_{[\la, \la+T^{-1}]}(P)\bigr\|_{L^{\frac{2(n+1)}{n+3}}(\Hn)\to L^{\frac{2(n+1)}{n-1}}(\Hn)}
\le C\la^{\frac{n-1}{n+1}}.
\end{equation}
\end{theorem}

By duality, \eqref{s1} is equivalent to 
\begin{equation}\label{s3}
T^{\frac12}\bigl\|\1_{[\la, \la+T^{-1}]}(P)\bigr\|_{L^{\frac{2(n+1)}{n+3}}(\Hn)\to L^2(\Hn)}
\le C\la^{\frac{n-1}{2(n+1)}}, \quad \la, T\ge1,
\end{equation}
and it is clear that this estimate along with \eqref{s1} yields \eqref{s2}.  Conversely,
by a standard $TT^*$ argument, \eqref{s2} implies \eqref{s1}.

Before proving the Theorem, let us make a couple more observations about these estimates.  First, by the spectral theorem and \eqref{s3}, we have that there is a uniform constant $C$ so that
\begin{multline}\label{s4}
T^{\frac12}\bigl\|m(T(P-\la))\bigr\|_{L^{\frac{2(n+1)}{n+3}}(\Hn)\to L^2(\Hn)}
\le C\|m\|_{L^\infty}\la^{\frac{n-1}{2(n+1)}}, \quad \la, T\ge1,
\\
\text{if } \, \, m\in C(\R) \, \, \, \text{and } \, \,  \text{supp }m \subset [0,1].
\end{multline}
From this and \eqref{s1} we immediately obtain the following result which will be useful in the sequel:

\begin{corr}\label{corr}
There is a constant $C$ which is independent of $\la, \, T\ge1$ so that if $m$ is as
in \eqref{s4} we have
\begin{equation}\label{s5}
T\bigl\|m(T(P-\la))\bigr\|_{L^{\frac{2(n+1)}{n+3}}(\Hn)\to L^{\frac{2(n+1)}{n-1}}(\Hn)}
\le C\|m\|_{L^\infty}\la^{\frac{n-1}{n+1}}.
\end{equation}
\end{corr}

To prove the theorem, let us first prove the special case corresponding to $T=1$, which we
can obtain by local techniques:

\begin{lemma}\label{localprop}
There is a uniform constant $C$ so that 
\begin{equation}\label{s6}
\|\1_{[\la,\la+1]}(P)\|_{L^2(\Hn)\to L^{\frac{2(n+1)}{n-1}}(\Hn)}\le C(1+\la)^{\frac{n-1}{2(n+1)}}, 
\, \, 
\la \ge 0.
\end{equation}
\end{lemma}

If $\la$ is bounded by a fixed constant, the estimate \eqref{s6} is a simple consequence of the Sobolev estimates for $\Hn$ that can be
found, for instance in \S 3 of \cite{Anker}, 
which say that if $1<p<q<\infty$ then for powers of the unshifted Laplacian,
$-\Delta_{\Hn}=P^2+(\tfrac{n-1}2)^2$, we have
\begin{equation}\label{Hsob}
\Bigl\|\bigl(P^2+(\tfrac{n-1}2)^2\bigr)^{-n(\tfrac1p-\tfrac1q)}\Bigr\|_{L^p(\Hn)\to L^q(\Hn)}<\infty.
\end{equation}
Therefore, in proving \eqref{s6} we shall assume that
$\la$ is large.  By duality and the spectral theorem, this then would follow from showing
that
$$\|\rho(\la-P)\|_{L^2(\Hn)\to L^{\frac{2(n+1)}{n-1}}(\Hn)}\le C\la^{\frac{n-1}{2(n+1)}}, \quad \la \gg 1,$$
assuming that 
\begin{equation}\label{s7}
\rho\in {\mathcal S}(\R), \, \, \, \rho(0)=1, \, \, \text{and } \, \,
\text{supp } \Hat\rho\subset [1/2,1].
\end{equation}
We then write
$$\rho(\la-P)=\frac1\pi \int \Hat \rho(t)e^{i\la t}\cos tP\, dt +\rho(\la+P).$$
Since by duality, the spectral theorem and the  Sobolev estimates \eqref{Hsob}, we have
$$\|\rho(\la+P)\|_{L^2(\Hn)\to L^{\frac{2(n+1)}{n-1}}(\Hn)}=O(\la^{-N}), \, \forall N, \,
\la \ge1,$$
it suffices to show that
$$\Bigl\|\int \Hat \rho(t)e^{i\la t}\cos tP\, dt\Bigr\|_{L^2(\Hn)\to L^{\frac{2(n+1)}{n-1}}(\Hn)}
=O(\la^{\frac{n-1}{2(n+1)}}), \, \, \la \ge 1.$$
Since the kernel of this operator vanishes when the distance between $x$ and $y$ is lager than one,
to prove this estimate it suffices to verify that
$$\Bigl\|\int \Hat \rho(t)e^{i\la t}\cos tP f\, dt\Bigr\|_{L^{\frac{2(n+1)}{n-1}}(B)}
\le C\la^{\frac{n-1}{2(n+1)}}\|f\|_{L^2(\Hn)}, \, \, \la \ge 1,$$
whenever $B$ is a geodesic ball in $\Hn$ of radius one.  If we choose geodesic normal
coordinates about the center, this follows from the proof of (5.1.3') in \cite{FIO},
which completes the proof of \eqref{s6}.

\begin{proof}[Proof of Theorem~\ref{HST}]
Since, as we noted before \eqref{s1} and \eqref{s2} are equivalent, it suffices to prove the former.
Repeating the first part of the proof of Lemma~\ref{localprop}, it suffices to show that
if $\rho$ is as in \eqref{s7} then 
$$T^{\frac12}\|\rho(T(\la-P))\|_{L^2(\Hn)\to L^{\frac{2(n+1)}{n-1}}(\Hn)}\le C\la^{\frac{n-1}{2(n+1)}}, \quad T, \la \ge 1.$$
If $\chi =|\rho|^2\in {\mathcal S}(\R)$, this follows from showing that
$$T\|\chi(\la-P)\|_{L^{\frac{2(n+1)}{n+3}}(\Hn)\to L^{\frac{2(n+1)}{n-1}}(\Hn)}\le 
C\la^{\frac{n-1}{n+1}}, \, \, \la, T\ge 1,$$
and since, by Sobolev estimates $\chi(T(\la+P))$ has $L^{\frac{2(n+1)}{n+3}}(\Hn)\to L^{\frac{2(n+1)}{n-1}}(\Hn)$ operator norm which is $O((\la T)^{-N})$ for every $N$, we would be done if we could
show that for $\la, T\ge1$ we have the uniform bounds
$$
\Bigl\| \int \Hat \chi(t/T) e^{i\la t}\cos tP\, dt\Bigr\|_{L^{\frac{2(n+1)}{n+3}}(\Hn)\to L^{\frac{2(n+1)}{n-1}}(\Hn)}\le C\la^{\frac{n-1}{n+1}}, \, \, \la , T\ge 1$$

If $\beta\in C^\infty_0(\R)$ satisfies $\beta(s)=1$, $|s|\le 1$, then it follows from
the Lemma~\ref{localprop}, duality and orthogonality that
$$
\Bigl\| \int \beta(t) \Hat \chi(t/T) e^{i\la t}\cos tP\, dt\Bigr\|_{L^{\frac{2(n+1)}{n+3}}(\Hn)\to L^{\frac{2(n+1)}{n-1}}(\Hn)}\le C\la^{\frac{n-1}{n+1}}, \, \, \la, T\ge 1.$$
Therefore, we would be done if we could show that
$$
\Bigl\| \int_0^\infty (1- \beta(t)) \Hat \chi(t/T) e^{i\la t}\cos tP\, dt\Bigr\|_{L^{\frac{2(n+1)}{n+3}}(\Hn)\to L^{\frac{2(n+1)}{n-1}}(\Hn)}\le C\la^{\frac{n-1}{n+1}}, \, \, \la, T\ge 1,$$
since the same argument will give this bound if the integral is taken over $(-\infty,0]$.

To prove this, as in Stein's argument for the Euclidean case, we shall use analytic interpolation.  Define the analytic family of operators
\begin{equation}\label{s8}
S^z_{\la,T}=\frac{e^{z^2}}{\Gamma(z+1)}\int_0^\infty t^z \, 
(1- \beta(t)) \Hat \chi(t/T) e^{i\la t}\cos tP\, dt.
\end{equation}
Then $S^0_{\la,T}$ is the operator in \eqref{s8}, and so, by Stein's analytic interpolation
theorem, we would have this estimate if we could show that
\begin{equation}\label{s9}
\|S^z_{\la,T}\|_{L^2(\Hn)\to L^2(\Hn)}\le C, \quad \text{Re }z=-1,\, \, \la, T\ge 1,
\end{equation}
as well as
\begin{equation}\label{s.10}
\|S^z_{\la,T}\|_{L^1(\Hn)\to L^\infty(\Hn)}\le C\la^{\frac{n-1}2}, \, \,
\text{Re }z=\tfrac{n-1}2, \, \la, T\ge 1.
\end{equation}

The estimate \eqref{s9} follows from the spectral theorem and the fact that the Fourier transforms
of $e^{z^2}t_+^z/\Gamma(z+1)$ 
are continuous functions whose $L^\infty$ norms which are bounded independent of $z$
if $\text{Re }z=-1$.

To prove \eqref{s.10}, let us first assume that $n$ is odd.  Then the kernel of
$S^{\frac{n-1}2 +i\sigma}_{\la, T}$, $\sigma \in \R$, is a constant $c_n$ times
\begin{equation}\label{s11}
\frac{e^{((n-1)/2+i\sigma)^2}}{\Gamma(\tfrac{n-1}2+i\sigma)}
\Bigl(\frac1{\sinh t}\frac{d}{dt}\Bigr)^{\frac{n-1}2}
\Bigl[(1-\beta(t))|t|^{\frac{n-1}2+i\sigma}\Hat \chi(t/T)e^{i\la t}\Bigr]\Biggl|_{t=d(x,y)},
\end{equation}
where $d(x,y)$ denotes the distance between $x$ and $y$ coming from the hyperbolic metric.
Since $(1-\beta(t))$ vanishes near the origin, and $t/\sinh t< 1$, $t>0$, it is easy to see
that this expression is bounded by a fixed multiple of $\la^{\frac{n-1}2}$ when $\la \ge 1$
and $\sigma \in \R$, which means that we have \eqref{s.10} when $n$ is odd.

To finish the proof, we have to establish \eqref{s.10}, when $n$ is even.  In this case
we can use the fact that the kernel is given by the formula
\begin{multline}\label{s.12}
c_n\frac{e^{((n-1)/2+i\sigma)^2}}{\Gamma(\tfrac{n-1}2+i\sigma)}
\int_t^\infty 
\frac{\sinh s}{\sqrt{\cosh s-\cosh t}} 
\\ \times \Bigl(\frac{1}{\sinh s}\frac{d}{ds}\Bigr)^{\frac{n}2}
\Bigl[(1-\beta(s))s^{\tfrac{n-1}2+i\sigma} \Hat \chi(s/T) e^{i\la s}\Bigr] \, ds,
\end{multline}
where, as before, $t=d(x,y)$.

To prove this we note that since $\beta$ equals one near the origin and
$\la, \, T\ge1$, if we use Leibniz's rule,
we can write the integral as
\begin{equation}\label{s.13}
\la^{\frac{n}2} \int_t^ \infty \frac{a(T,\sigma,\la; s)}{\sqrt{1+e^{-2s}-e^{t-s}-e^{-t-s}}} \,
e^{i\la s} \, ds,
\end{equation}
where for constants which are independent of $\sigma\in \R$ and $T, \la \ge 1$ we have
\begin{equation}\label{s.14}
|a(T,\sigma,\la;s)|\, + \, |\partial_s a(T,\sigma,\la;s)|\le C(1+s)^{-2},
\, \, \text{if } \, 1-\beta(s)\ne 0, \, \, s>0.
\end{equation}
Note also that we have that for $1-\beta(s)\ne 0$ we have
\begin{equation}\label{s.15}
1+e^{-2s}-e^{t-s}-e^{-t-s} \ge c(s-t), \quad \text{if } \, s\in [t,t+1], \, t\ge 0,
\end{equation}
for some fixed $c>0$.
Using this estimate and the bound for the first term in the left of \eqref{s.14}, we
deduce that
$$
\Bigl|\la^{\frac{n}2}\int_{t}^{t+\la^{-1}} \frac{a(T,\sigma;s)(1-\beta(s))}{
\sqrt{1+e^{-2s}-e^{t-s}-e^{-t-s}}} e^{i\la s}\, ds \Bigr|\le C\la^{\frac{n-1}2},
$$
as posited in \eqref{s.10},
independent of $\sigma, T$ and $\la$ as above.  To handle the remaining piece,
we note that we also have that there also must be a fixed $c>0$ so that if $1-\beta(s)\ne 0$ and $t>0$
then
$$
1+e^{-2s}-e^{t-s}-e^{-t-s}\ge c, \quad \text{if } \, \, s\ge t+1.
$$
Using this bound as well as \eqref{s.14}-\eqref{s.15}, we conclude that the remaining
piece of \eqref{s.13} must be bounded independent of $\sigma, T,\la$ as above since
after integrating by parts it is dominated by
\begin{multline*}
\la^{\frac{n}2-1}\Bigl|\frac{a(T,\sigma;s)(1-\beta(s))}{
\sqrt{1+e^{-2s}-e^{t-s}-e^{-t-s}}}\Bigr|_{s=t+\la^{-1}}
\\
+\la^{\frac{n}2-1}\int_{t+\la^ {-1}}^\infty \Bigl| \frac\partial{\partial s}
\Bigl(\frac{a(T,\sigma;s)(1-\beta(s))}{
\sqrt{1+e^{-2s}-e^{t-s}-e^{-t-s}}}\Bigr)\Bigr| \, ds
=O(\la^{\frac{n-1}2}).
\end{multline*}

\end{proof}

\newsection{Uniform $(L^r,L^s)$ resolvent bounds for $\Hn$}

The goal of this section is to prove the uniform resolvent bounds
\begin{equation}\label{r.1}
\bigl\| (\Delta_{\Hn}+(\tfrac{n-1}2)^2+z^2)^{-1}\bigr\|_{L^r(\Hn)\to L^s(\Hn)}
\le C, \quad z\in \C\backslash\R_+,\, \, |z|\ge 1,
\end{equation}
assuming that, as in \cite{KRS},
\begin{equation}\label{r.2}
n\bigl(\frac1r-\frac1s\bigr)=2, \, \, \text{and } \, \,
\min \bigl(\bigl|\frac1r-\frac1s\bigr|, \, \bigl|\frac1s-\frac12\bigr|\bigr)>
\frac1{2n}.
\end{equation}
This is the first part of Theorem~\ref{negresolve}.  By the same argument that \eqref{i.9} implies \eqref{i.10}, one sees that \eqref{r.1}
implies the other part of the theorem, \eqref{i.12}.

Clearly, by letting $-\kappa\nearrow 0$, one sees that this inequality
implies the earlier Euclidean estimates \eqref{i.2} in \cite{KRS},
and based on this, one sees that \eqref{r.2} is the sharp range
of exponents for all of these estimates.
  Pictorially, we have \eqref{r.2}
if $(\frac1r,\frac1s)$ is on the open line segment connecting $\alpha$ and $\alpha'$ in
Figure~\ref{fig2}.

Note also, that if we write $z=\la+i\mu$, then it suffices to verify that we have
\eqref{r.1} if
\begin{equation}\label{r.3}
z=\la+i\mu, \quad \text{with } \, \la \ge 1, \, \, \mu\ne 0,
\end{equation}
since the remaining cases of \eqref{r.1} follow from Sobolev estimates.

To prove \eqref{r.1}, as in \cite{BSSY}, we shall use the formula
\begin{equation}\label{r.4}
\bigl(\Delta_{\Hn}+(\tfrac{n-1}2)^2+(\la+i\mu)^2\bigr)^{-1}
=\frac{\sgn \mu}{i(\la+i\mu)}\int_0^\infty e^{i(\sgn \mu)\la t}
e^{-|\mu|t} \, (\cos tP)\, dt,
\end{equation}
where now $P$ is the square root of the shifted 
Laplacian on $\Hn$, i.e., $P=\sqrt{-\Delta_{\Hn}-(\tfrac{n-1}2)^2}$.  Thus
$$u(t,\cdot)=\cos tP$$
solves the Cauchy problem for the shifted Laplacaian
$$(\partial_t^2-\Delta_{\Hn}-(\tfrac{n-1}2)^2)u(t,x)=0, 
\quad u(0,\cdot)=f, \, \, \, \partial_tu(0,\cdot)=0.$$

To prove this,  choose
a Littlewood-Paley bump function $\beta\in C^\infty_0((1/2,2))$ satisfying
$$\sum_{j\in \Z} \beta(2^{-j}t)=1, t>0.$$
It then follows that
$$\beta_0(t)=1-\sum_{j=0}^\infty \beta(2^{-j}t)$$
equals one for $t>0$ near the origin and vanishes when $t\ge2$.  
Thus, the kernel of 
\begin{equation}\label{4.5}S_0=\frac{\sgn \mu}{i(\la+i\mu)}\int_0^\infty \beta_0(t) \, e^{i(\sgn \mu)\la t}
e^{-|\mu|t} \, (\cos tP)\, dt,
\end{equation}
vanishes when $d_{\Hn}(x,y)\ge 2$.  As we shall see, its kernel
is similar to the corresponding local operator $R^{\la, \mu}_0$ that
we encountered in our bounds for $S^n$.  Specifically, in the appendix
we shall prove the following

\begin{proposition}\label{prop4.1}
Let $S^{\la,\mu}_0(x,y)=S_0(x,y)$ denote the kernel of the operator
in \eqref{4.5}.  Then
\begin{equation}\label{4.6}
S_0(x,y)=\sum_\pm a_\pm(\la; x,y)e^{\pm i\la d_{\Hn}(x,y)} +O((d_{\Hn}(x,y))^{2-n}),
\end{equation}
where
\begin{equation}\label{4.7}
a_\pm(\la; x,y)=0 \quad \text{if } \, d_{\Hn}(x,y)\ge 2,
\end{equation}
and, moreover, with constants independent of $\la\ge 1$
\begin{equation}\label{4.8}
|a_\pm(\la; x,y)|\le C\bigl(d_{\Hn}(x,y)\bigr)^{2-n}, \quad\text{if } \,
d_{\Hn}(x,y)\ge \la^{-1},
\end{equation}
and
\begin{equation}\label{4.9}
|\nabla_x^{\alpha_1}\nabla^{\alpha_2}_ya_\pm(\la;x,y)|\le C_{\alpha_1,\alpha_2}\la^{\frac{n-3}2}\bigl(d_{\Hn}(x,y)\bigr)^{-\frac{n-1}2-\alpha_1-\alpha_2}, \quad\text{if } \, 
d_{\Hn}(x,y)\ge \la^{-1}.
\end{equation}
\end{proposition}

Due to this Proposition, it is clear that we can use the proof
of \eqref{2.35} to show that if $r,s$ are as in \eqref{i.1} then there is
a constant $C_{r,s}$ so that for all $\la\ge 1$ then
\begin{equation}\label{4.10}
\|S_0f\|_{L^s(\Hn)}\le C_{r,s}\|f\|_{L^r(\Hn)}.
\end{equation}
One first argues that whenever $f$ is supported in a unit ball
this bound holds with constants uniform of the center, and,
by \eqref{4.7}, this implies
\eqref{4.10} since $\Hn$ has bounded geometry.


Based on \eqref{4.10}, we would have \eqref{r.1} if we could show that there
is a uniform constant $C$ so that
if 
$$S_k=\frac{\sgn \mu}{i(\la+i\mu)}\int_0^\infty \beta(2^{-k}t) \, e^{i(\sgn \mu)\la t}
e^{-|\mu|t} \, (\cos tP)\, dt,$$
then if $\la\ge 1$
\begin{equation}\label{r.6}
\bigl\| S_k\bigr\|_{L^r(\Hn)\to L^s(\Hn)}
\le C2^{-k}, \, \, k=1,2,\dots, 
\, \, \, n(\tfrac1r-\tfrac1s)=2, \, \, \tfrac{2n}{n+3}\le r\le \tfrac{2n}{n+1}.
\end{equation}
Since this implies that the non-local part of the resolvent is actually bounded
for exponents $(\tfrac1r,\tfrac1s)$ on the {\em closed} segment joining
$\alpha$ and $\alpha'$ in the Figure~\ref{fig2}, i.e., we have
\begin{multline}\label{r.7}
\bigl\| \, (\Delta_{\Hn}+(\tfrac{n-1}2)^2+(\la+i\mu)^2)^{-1} -S_0\bigr\|_{L^r(\Hn)\to L^s(\Hn)}
\le C, \quad \la, \mu \in \R, \, \la \ge 1, 
\\
\text{and } \, n(\tfrac1r-\tfrac1s)=2, \, \, \tfrac{2n}{n+3}\le r\le \tfrac{2n}{n+1}.
\end{multline}

To prove \eqref{r.6}, we shall use an interpolation argument.  The three ingredients
we require are that there is a uniform constant $C$ so that for $k=1,2,\dots$
\begin{equation}\label{r.8}
\|S_k\|_{L^2(\Hn)\to L^{\frac{2(n+1)}{n-1}}(\Hn)}\le C2^{\frac{k}2}
\la^{-\frac{n+3}{2(n+1)}},
\end{equation}
 as well as for all $k,N\in {\mathbb N}$, there is a constant $C_N$, depending only on $N$ so that
 \begin{equation}\label{r.9}
 \|S_k\|_{L^1(\Hn)\to L^\infty(\Hn)}\le C_N 2^{-kN}\la^{\frac{n-3}2},
 \end{equation}
 and
 \begin{equation}\label{r.10}
 \|S_k\|_{L^{\frac{2n}{n+1}}(\Hn)\to L^\infty(\Hn)}\le C_N 2^{-kN}
 \la^{\frac{n-3}2}.
 \end{equation}

Since $\la^{-\frac{n+3}{2(n+1)}}=\la^{-1+\frac{n-1}{2(n+1)}}$, \eqref{r.8}
follows from the formula for $S_k$ and \eqref{s4}.

In odd dimensions, both \eqref{r.9} and \eqref{r.10} follow immediately from the fact that the kernel of $S_k$ is given by
$$c_n\frac{\sgn \mu}{i(\la+i\mu)}\Bigl(\frac1{\sinh t}\frac{d}{dt}\Bigr)^{\frac{n-1}2}
\Bigl[ \beta(2^{-k}t) \, e^{i(\sgn \mu)\la t}e^{-|\mu|t}\Bigr],
\quad t=d_{\Hn}(x,y).$$
In even dimensions, one establishes these two bounds using the fact that the kernel
is given by 
\begin{multline*}c_n\frac{\sgn \mu}{i(\la+i\mu)}\int_t^\infty 
\frac{\sinh s}{\sqrt{\cosh s-\cosh t}} 
\\ \times \Bigl(\frac{1}{\sinh s}\frac{d}{ds}\Bigr)^{\frac{n}2}
\Bigl[\beta(2^{-k}s) \, e^{i(\sgn \mu)\la s}e^{-|\mu|s}\Bigr] \, ds,
\quad t=d_{\Hn}(x,y),
\end{multline*}
by the same argument that established \eqref{s.10} from \eqref{s.12}.

Since
$$\frac{(n-1)^2}{2n(n+1)}=\frac{n-1}{2(n+1)}\cdot  \frac{n-1}{n},$$
and
$$-\frac2{n+1}=\frac{n-3}2 \cdot \frac1n-
\frac{n+3}{2(n+1)}\cdot \frac{n-1}n,$$
if we interpolate between \eqref{r.8} and \eqref{r.9},
we conclude that
\begin{equation}\label{r.11}
\|S_k\|_{L^{\frac{2n}{n+1}}(\Hn)\to L^{\frac{2n(n+1)}{(n-1)^2}}(\Hn)}
=O_N(2^{-kN}\la^{-\frac2{n+1}}).
\end{equation}
This is a bound which corresponds to the point $A$ in the figure.

To get the bound \eqref{r.6} corresponding to one the endpoints 
\begin{equation}\label{r.12}
\|S_k\|_{L^{\frac{2n}{n+1}}(\Hn)\to L^{\frac{2n}{n-3}}(\Hn)}\le C2^{-k}
\end{equation}
which corresponds to the point
$\alpha = (\tfrac{n+1}{2n},\tfrac{n-3}{2n})$ in the figure
we need to interpolate between \eqref{r.11} and \eqref{r.10}.
We first note that
$$\frac{n-3}{2n}=\frac{(n-1)^2}{2n(n+1)}\cdot \theta,$$
with 
$$\theta =\frac{(n+1)(n-3)}{(n-1)^2}.$$
Since for this $\theta$, a calculation shows that
$$0=\frac{n-3}2\cdot (1-\theta)-\frac2{n+1}\cdot \theta,$$
we conclude that \eqref{r.12} does indeed follow from \eqref{r.10} and \eqref{r.11} via interpolation.

From this we obtain all of \eqref{r.6}, since by duality we have from \eqref{r.12} that
\begin{equation}\label{r.13}
\|S_k\|_{L^{\frac{2n}{n+3}}(\Hn)\to L^{\frac{2n}{n-1}}(\Hn)}\le C2^{-k},
\end{equation}
which corresponds to the point $\alpha'$ in the figure and yields \eqref{r.6} for the remaining exponents if we interpolate with \eqref{r.12}.

%
%

\subsubsection{Improved estimates for ${\mathbb H}^3$}

Let us now give the simple argument showing that in three dimensions we can obtain the following improvement
over Theorem~\ref{negresolve}:
\begin{equation}\label{4.19}
\|u\|_{L^s({\mathbb R}^3, dV_{-k})}\le C_{r,s}
\bigl\|\bigl((\Delta_{-\kappa}+\sqrt\kappa +\zeta\bigr)u\bigr\|_{L^r({\mathbb R}^3,dV_{-\kappa})}, 
\, \, \, u\in C^\infty_0, \, \, \zeta \in \C.
\end{equation}
By a straightforward variant of the scaling argument at the end of \S2, these bounds which are uniform
both in $\zeta$ and in the curvature $-\kappa$, $\kappa>0$, would just follow from the special case
where the curvature is $-1$, i.e.,
\begin{equation}\label{4.20}
\|u\|_{L^s({\mathbb H}^3,dV_{{\mathbb H}^3})}\le C_{r,s}\bigl\|\bigl((\Delta_{{\mathbb H}^3}+1+\zeta\bigr)u\bigr\|_{L^r({\mathbb H}^3,dV_{{\mathbb H}^3})}, 
\, \, \, u\in C^\infty_0, \, \, \zeta\in \C.
\end{equation}
In proving this, we may assume that $\zeta=(\la+i\mu)^2$ with $\la\in \R$, $\mu>0$, and then we just use the fact
that the kernel of $\bigl((\Delta_{{\mathbb H}^3}+1+z^2\bigr)^{-1}$ equals
$$\frac{1}{4\pi }\frac1{\sinh (d_{{\mathbb H}^3}(x,y))} e^{(i\la -\mu)d_{{\mathbb H}^3}(x,y)}$$
(see \cite[p. 105]{Taylor}).  Because of this, we obtain \eqref{4.20} via the Hardy-Littlewood-Sobolev inequality
and Young's inequality.

\newsection{Appendix: Proof of the Propositions}

We shall conclude matters by proving Propositions \ref{prop2.1}, \ref{prop2.2}, \ref{prop2.4} and \ref{prop4.1}.  As we noted before,
each is essentially in the literature (e.g., \cite{BSSY}, \cite{DKS},
\cite{KRS}, \cite{SY}, \cite{Sthesis},
\cite{Sef}, and \cite{FIO}).

\subsubsection{Proof of Proposition~\ref{prop2.1}}
Let us start with Proposition~\ref{prop2.1}, which concerns
asymptotics for the kernel of projection onto spherical harmonics of degree $k$,
which involve multiples of the zonal functions on $S^n$.  As we noted,
before, the asymptotics that we require are essentially in the classical
Darboux formula (see \cite{Szego}).

To prove that $H_k(x,y)$ satisfies \eqref{2.11}--\eqref{2.14}, we first
note that the first bound is just a special case of sup-norm
estimates for spectral clusters.  See, e.g. \cite[(3.2.5)--(3.2.6)]{SHang}.

To obtain the off-diagonal assertions in \eqref{2.12}--\eqref{2.14}, we note
that we can write, with $P=\sqrt{-\Delta_{S^n}+(\tfrac{n-1}2)^2}$,
and $\la_k=k+\tfrac{n-1}2$
\begin{multline}\label{5.1}
H_k(x,y)=\frac1{2\pi}\int_{-\infty}^\infty \Hat \rho(t)
e^{i\la_kt} e^{-itP}\, dt
\\
=\frac1\pi \int_{-\infty}^\infty \Hat \rho(t)
e^{i\la_kt} \bigl(\cos tP\bigr)(x,y)\, dt +\rho(\la_k+P)(x,y),
\end{multline}
if $\rho\in C^\infty_0(\R)$ satisfies
$$\rho(\tau)=1, |\tau|\le 1/2, \quad \rho(\tau)=0, \, \, |\tau|\ge 3/4.$$
Since the last term and all of its derivatives or $O_N((1+\la_k)^{-N})$,
for every $N$, it suffices to show that
\begin{equation}\label{5.1}
\tilde H_k(x,y)= \int_{-\infty}^\infty \Hat \rho(t)
e^{i\la_kt} \bigl(\cos tP\bigr)(x,y)\, dt
\end{equation}
is as in \eqref{2.12}--\eqref{2.14}.

If $n$ is odd then $\cos tP$ is $2\pi$-periodic.  Since $\Hat \rho\in 
{\mathcal S}(\R)$, it then follows that if we set
$$\psi_{odd}(t)=\sum_{j\in \Z}\Hat \rho(t-2\pi j),$$
then $\psi_{odd}$ is smooth and $2\pi$-periodic, and, therefore,
\begin{equation}\label{5.3}
\tilde H_k(x,y)=\int_{-\pi}^{\pi} \psi_{odd}(t)e^{i\la_kt}
\, \bigl(\cos tP\bigr)(x,y) \, dt, \quad n \, \text{odd}.
\end{equation}
Similarly, since $\cos tP$ is $4\pi$-periodic when $n$ is even,
if we set 
$$\psi_{even}(t)=\sum_{j\in \Z}\Hat \rho(t-4\pi j),$$
then $\psi_{even}$ is smooth and $4\pi$ periodic and we have
\begin{equation}\label{5.4}
\tilde H_k(x,y)=\int_{-2\pi}^{2\pi} \psi_{even}(t)e^{i\la_kt}
\, \bigl(\cos tP\bigr)(x,y) \, dt, \quad n \, \text{even}.
\end{equation}

To proceed, let us first assume that $n$ is odd.  We then fix
$\eta\in C^{\infty}_0(\R)$ satisfying
$$\eta(t)=1, \, \, |t|\le 1/2, \quad \text{and } \, \, 
\eta(t)=0, \, \, |t|\ge 1.$$
We then can write
\begin{multline*}\tilde H_k(x,y)=\int_{-\pi}^{\pi} \eta(\pi-|t|)\psi_{odd}(t)e^{i\la_kt}
\, \bigl(\cos tP\bigr)(x,y) \, dt
\\
+\int_{-\pi}^{\pi}(1-\eta(\pi-|t|)) \psi_{odd}(t)e^{i\la_kt}
\, \bigl(\cos tP\bigr)(x,y) \, dt.\end{multline*}
The proof of \cite[Lemma 5.1.3]{FIO} shows that the first term
is as in \eqref{2.12}--\eqref{2.13}, 
and, since for odd $n$ $(\cos tP)(x,y)=0$ if $d_{S^n}(x,y)\ne |t|$, $0<|t|<\pi$
(see \cite{Taylor}) the second term vanishes
when $d_{S^n}(x,y)\le \pi-1$, which means that for odd dimensions
we have all but \eqref{2.14} in Proposition~\ref{prop2.1}.  Since
\eqref{2.14} just follows from what we have done and the fact
that $H_k(x,y^*)=(-1)^kH_k(x,y)$, the proof of Proposition~\ref{prop2.1}  for odd $n$ is complete.

The proof for even $n$ is similar.  In this case one splits
\begin{align*}
\tilde H_k(x,y)&=\int_{-2\pi}^{2\pi} \eta(\pi-t)\, \psi_{even}(t)
e^{i\la_kt}\bigl(\cos tP\bigr)(x,y) \, dt
\\
&+\int_{-2\pi}^{2\pi} \eta(\pi+t)\, \psi_{even}(t)
e^{i\la_kt}\bigl(\cos tP\bigr)(x,y) \, dt
\\
&+\int_{-2\pi}^{2\pi}\bigl(1- \eta(\pi-t)-\eta(\pi+t)\bigr) \, \psi_{even}(t)
e^{i\la_kt}\bigl(\cos tP\bigr)(x,y) \, dt,
\end{align*}
and uses the fact that
$$\bigl(\cos tP\bigr)(x,y)=-\bigl(\cos (t+2\pi)P\bigr)(x,y)$$
and
$$\bigl(\cos tP\bigr)(x,y) \, \, \, \text{is smooth } \, \,
\, \text{if } \, \, d_{S^n}(x,y)\ne |t| \, \, \text{mod }\pi.$$
By the latter fact, the first two terms are smooth with derivatives
bounded independent of $\la_k$ if $d_{S^n}(x,y)\le 3\pi/4$.  Using these facts one can also use the proof of \cite[Lemma 5.1.3]{FIO} to see
that, under this assumption, the last term in the decomposition
is as in \eqref{2.12}--\eqref{2.13}.  This implies the first part
of Proposition~\ref{prop2.1} for even $n$, and since, as before,
\eqref{2.14} trivially follows from this, the proof is complete. \qed 

Since the proof of Proposition~\ref{prop2.4} is similar to the above, let us now turn to it.

\subsubsection{Proof of Proposition~\ref{prop2.4}}

Recall that for $\la \ge 1$ the Euclidean resolvent kernels
\begin{align}\label{5.5}
\frac{\text{sgn }\mu}{i(\la+i\mu)}\int_0^\infty &e^{i(\text{sgn }\mu)\la t}
e^{-|\mu| t}\bigl(\cos t\sqrt{-\Delta_{\Rn}}\bigr)(x) \, dt
\\
&=\frac{(2\pi)^{-n}\text{sgn }\mu}{i(\la+i\mu)}
\int_0^\infty 
\int_{\Rn}
e^{i(\text{sgn }\mu)\la t}
e^{-|\mu| t} \, \cos (t|\xi|) \, e^{ix\cdot \xi } \, d\xi dt \notag 
\\
&=(2\pi)^{-n}\int_{\Rn} \frac{e^{ix\cdot \xi}}{-|\xi|^2+(\la +i\mu)^2}\, d\xi
=K_{\la,\mu}(|x|), \notag 
\end{align}
can be written as
\begin{equation}\label{5.6}
K_{\la,\mu}(|x|)=\sum_\pm a_\pm(\la; |x|)e^{\pm i\la |x|} +O(|x|^{2-n}),
\end{equation}
where
\begin{equation}\label{5.7}
|\partial^j_ra_\pm (\la; r)|\le C_j\la^{\frac{n-3}2}r^{-\frac{n-1}2 - j}, \quad
r\ge \la^{-1},
\end{equation}
and
\begin{equation}\label{5.8}
K_{\la,\mu}(|x|)|\le C|x|^{2-n}, \quad |x|\le \la^{-1},
\end{equation}
where the constants in \eqref{5.7}--\eqref{5.8} are independent of $\mu\in R$ and
$\la\ge 1$.  This follows easily from stationery phase, and it also follows from writing
the kernel in terms of Bessel potentials.  See e.g., \cite[p. 338--339]{KRS}.

To prove that when 
$P=\sqrt{-\Delta_{S^n}+(\tfrac{n-1}2)^2}$
$$ 
\frac{\text{sgn }\mu}{i(\la+i\mu)}\int_0^\infty \rho(t) e^{i(\text{sgn }\mu)\la t-|\mu|t} \, \bigl(\cos tP\bigr)(x,y) \, dt
$$
has similar behavior if $0\le \rho\in C^\infty_0(\R)$ satisfies
$\rho(t)=1$, $|t|\le 1/2$ and $\rho(t)=0$, $|t|\ge 1$, we shall use
the Hadamard parametrix (see \cite{Hor3}, \cite{SHang}) which says
that we can write for say, $|t|\le 1$,
$$
\bigl(\cos tP\bigr)(x,y)=(2\pi)^{-n}\int_{\Rn} e^{id_{S^n}(x,y)\1 \cdot
\xi}\alpha(t,x,y;|\xi|) \, \cos t|\xi| \, d\xi,
 $$
 modulo a smooth function, where
 $$ \1 =(1,0\dots,0),$$
and $\alpha \in S^0$ (zero order symbol) satisfies
\begin{equation}\label{5.9}
|\partial^\beta_{t,x,y}\partial_r^j\alpha(t,x,y;r)|\le C_{\beta,j}(1+r)^{-j}.
\end{equation}
In fact, modulo a symbol of order -2, $\alpha$ just equals a smooth
function of $x$ and $y$ which is independent of $|\xi|$.  Plugging this
into \eqref{5.5}, it is easy to see that the top order part of the 
parametrix contributes to a term satisfying the analog of \eqref{5.6}-\eqref{5.8} with $|x|$ being replaced by $d_{S^n}(x,y)$.  The smooth
error term in the parametrix clearly contributes to a term which is
$O(1)$ with bounds independent of $\la\ge 1$.

Let us now give the argument that the full parametrix also gives rise
to a term satisfying the analog of \eqref{5.6}-\eqref{5.8}.  In other
words, if $\alpha$ is as in \eqref{5.9}, we shall show that
when $\la\ge1$
\begin{equation}\label{5.10}
\frac1{\la+i\mu}\int_0^\infty \int_{\Rn}e^{id_{S^n}(x,y)\1\cdot \xi}
\rho(t) e^{i(\text{sgn }\mu)\la t-|\mu|t} \alpha(t,x,y;|\xi|)
\, \cos t|\xi| \, d\xi dt,
\end{equation}
is as in \eqref{5.6}-\eqref{5.8}.  

We first choose $\beta\in C^\infty_0((1/4,2)$ satisfying
$$\beta(r)=1, \quad 1/2\le r\le 3/2.$$
Then, by a simple integration by parts argument, the difference between
\eqref{5.10} and
\begin{equation}\label{5.11}\frac1{\la+i\mu}
\int_0^\infty \int_{\Rn}e^{id_{S^n}(x,y)\1\cdot \xi}
\rho(t) e^{i(\text{sgn }\mu)\la t-|\mu|t} \beta(|\xi|/\la) \,  \alpha(t,x,y;|\xi|)
\, \cos t|\xi| \, d\xi dt,
\end{equation}
is $O((d_{S^n}(x,y))^{2-n})$, independent of $\la \ge1$.  Also, it is
straightforward to see that \eqref{5.11} is $O(\la^{n-2})$, and so
we only need to check that this term is of the desired form
when $d_{S^n}(x,y)\ge \la^{-1}$.

To finish the proof of Proposition~\ref{prop2.4} and show that 
\eqref{5.11} has the desired form, we recall that the Fourier transform
of surface measure on the $(n-1)$-sphere can be written for $|x|\ge1$ as
$$\int_{S^{n-1}}  e^{i\omega\cdot x} \, dS(\omega)
=|x|^{-\frac{n-1}2}\sum_\pm a_\pm(|x|)e^{i\pm |x|},$$
where
$$ |\partial_r^ja_\pm(r)|\le Cr^{-j}, \quad r\ge 1. $$
Writing $\xi=r\omega$ in polar coordinates, we find that for
$d_{S^n}(x,y)\ge \la^{-1}$ we have that \eqref{5.11} can be rewritten
as the sum over $\pm$ of
\begin{multline*}\frac{\la^{\frac{n+1}2}}{\la+i\mu}
\bigl(d_{S^n}(x,y)\bigr)^{-\frac{n-1}2} 
\int_0^\infty \int_0^\infty \beta(r)\rho(t)
e^{i(\text{sgn }\mu)\la t-|\mu|t}e^{\pm i\la r d_{S^n}(x,y)} 
\cos(\la tr) \, r^{\frac{n-1}2} \, dr dt
\\
=\frac{\la^{\frac{n+1}2}}{\la+i\mu} \, \frac{e^{\pm i\la d_{S^n}(x,y)}}{(d_{S^n}(x,y))^{\frac{n-1}2}} \times b_\pm(\la; d_{S^n}(x,y)),
\end{multline*}
where
$$b_\pm(\la; \tau)=  \int_0^\infty \int_0^\infty \beta(r)\rho(t)
e^{i(\text{sgn }\mu)\la t-|\mu|t}e^{\pm i\la (r-1) \tau} 
\cos(\la tr) \, r^{\frac{n-1}2} \, dr dt.$$
Since clearly
$$ b=O(\la^{-1}),$$
and, moreover, a simple integration by parts argument shows that
$$|\partial_\tau^j b(\la; \tau)|\le C_j \la^{-1}\tau^{-j}, \quad \tau \ge \la^{-1}, $$
we conclude that \eqref{5.11} has the desired form, which completes the
proof. \qed

\subsubsection{Proof of Proposition~\ref{prop4.1}}

Since $\Hn$ has bounded geometry it is clear that the proof of Proposition~\ref{prop2.4} just given can be used to show that the kernel of the local
resolvent operator in \eqref{4.5} satisfies 
\eqref{4.6}-\eqref{4.9}.  Note that \eqref{4.7} just follows from Huygens principal.

Alternately, in odd dimensions one could just use the fact that
$$S_0(x,y)=c_n\frac{\text{sgn }\mu}{i(\la+i\mu)}
\Bigl(\frac1{\sinh t}\frac{d}{dt}\Bigr)^{\frac{n-1}2}
\Bigl[\beta_0(t)e^{i(\text{sgn }\mu )\la t -|\mu|t}\Bigr]_{t=d_{\Hn}(x,y)},$$
while in even dimensions one can reach the conclusions of Proposition~\ref{prop4.1}
by using the fact that 
\begin{multline*}
S_0(x,y)=c_n\frac{\sgn \,  \mu}{i(\la+i\mu)}\int_t^\infty 
\frac{\sinh s}{\sqrt{\cosh s-\cosh t}} 
\\ \times \Bigl(\frac{1}{\sinh s}\frac{d}{ds}\Bigr)^{\frac{n}2}
\Bigl[\beta_0(s) \, e^{i(\sgn \mu)\la s-|\mu|s}\Bigr] \, ds,
\quad t=d_{\Hn}(x,y).
\end{multline*}
\qed 

\subsubsection{Proof of Proposition~\ref{prop2.2}}  We need to show that
$$\|T_\la f\|_{L^s(B_1(0))}\le C_{r,s}\|f\|_{L^r(B_1(0))}, \quad \text{supp }f\subset B_1(0),$$
assuming that
$$T_\la f(x)=\int e^{i\la d_g(x,y)}a(x,y)f(y)\, dy,$$
where $a(x,y)$ vanishes if $d_g(x,y)>4$, and, moreover,
\begin{equation}\label{5.12}
|\partial^\alpha_{x,y}a(x,y)|\le C_\alpha \la^{\frac{n-3}2}\bigl(d_g(x,y)\bigr)^{-\frac{n-1}2-|\alpha|},
\quad d_g(x,y)\ge \la^{-1},
\end{equation}
assuming that
\begin{equation}\label{5.13}
|a(x,y)|\le C\bigl(d_g(x,y)\bigr)^{2-n}, \quad d_g(x,y)\le \la^{-1},
\end{equation}
and $r,s$ are as in \eqref{i.1}.  We are also assuming that $g$ is a smooth metric on $\Rn$ with injectivity radius
10 or more, and $B_r(x)$ denotes the geodesic ball of radius $r$ centered at $x$ if, say, $0<r<10$.

To prove these bounds, as before, choose $\beta\in C^\infty_0((1/2,2))$ satisfying
$\sum_{j\in \Z}\beta(2^{-j}t)=1$, $t>0$.  We then set
$$\bigl(T^k_\la f\bigr)(x)=\int e^{i\la d_g(x,y)}\beta(\la 2^{-k}d_g(x,y)) \,
a(x,y)f(y)\, dy,\quad k=1,2,3,\dots .$$
Then if 
$$T^0_\la =T_\la -\sum_{k=1}^\infty T^k_\la ,$$
it follows from \eqref{5.12}-\eqref{5.13}, the Hardy-Littlewood-Sobolev inequality and the fact
that $n(\tfrac1r-\tfrac1s)=2$ that
$$\|T^0_\la f\|_{L^s}\le C_{r,s}\|f\|_{L^r}.$$

Therefore, we would be done if we could show that whenever $r, s$ are as in \eqref{i.1} there is a
$$\sigma_{r,s}>0$$
so that
\begin{equation}\label{5.14}
\|T^k_\la f\|_{L^s}\le C_{r,s} 2^{-k\sigma_{r,s}}\|f\|_{L^s}.
\end{equation}
Since the kernel of $T^k_\la$ vanishes when $d_g(x,y)\notin [\la^{-1}2^{k-1},\la^{-1}2^{k+1}]$, in proving
this we may assume that
\begin{equation}\label{5.15}
\text{supp }f \subset B_{\la^{-1}2^k}(0).
\end{equation}
After a simple change of scale argument, we see that this is equivalent to showing that
\begin{equation}\label{5.16}
\|S^k_\la f\|_{L^s}\le C_{r,s}2^{-k\sigma_{r,s}}\|f\|_{L^s},
\end{equation}
if
$$S^k_\la f(x)=2^{\frac{n-3}2 k}\int e^{i2^k\phi_k(x,y)}b_k(x,y) f(y)\, dy,$$
where
$$\phi_k(x,y)=\e^{-1}d_g(\e x,\e y), \quad \e=\la^{-1}2^k,$$
and
$$b_k(x,y)=2^{\frac{n-1}2 k}\la^{-(n-2)} \, \beta(\phi_k(x,y)) \, a(\la^{-1}2^kx,\la^{-1}2^ky).$$

Note that if $g_\e$ is the ``stretched" metric $g_{ij}(\e x)$ then
$\phi_k(x,y)$ is just the Riemannian distance funtion for this metric with
$\e$ as above, i.e.,
$$\phi_k(x,y)=d_{g_\e}(x,y), \quad \e =\la^{-1}2^k.$$
The amplitude $b_k$ then vanishes if $d_{g_\e}(x,y)\notin [1/2,2]$, and \eqref{5.15} is equivalent to
$\text{supp }f\subset B_{1}(0)$ for this metric if $\e=\la^{-1}2^k$.  
Also, by \eqref{5.12}, the amplitudes satisfy
$$|\partial_{x,y}^\alpha b_k(x,y)|\le C_\alpha,$$
for every multi-index $\alpha$.
Thus, by the special case of
Stein's oscillatory integral theorem~\cite{St}, \eqref{2.26}, we have that
$$\|S^k_\la f\|_{L^q}\le C2^{\frac{n-3}2k-\frac{n}q k}\|f\|_{L^p}, \quad
\text{if } \, 1\le p\le 2, \, \, q=\tfrac{n+1}{n-1}p'.$$
Since the stretched metrics $g_\e$ tend to the constant coefficient metric $g_{ij}(0)$ as
$\e\searrow 0$, it is clear that we can choose the constant $C$ to be independent of $k$.
In particular, if as before
$$p=\tfrac{2n}{n+1}, \quad q=\tfrac{2n(n+1)}{(n-1)^2},$$
which corresponds to the point $A$ in Figure~\ref{fig2}, we have 
$$\|S^k_\la f\|_{L^{\frac{2n(n+1)}{(n-1)^2}}}\le C2^{\frac{n-3}2k-\frac{(n-1)^2}{2(n+1)} k}\|f\|_{L^{\frac{2n}{n+1}}}
=C2^{-\frac{2k}{n+1}}\|f\|_{L^{\frac{2n}{n+1}}}.$$
We also, of course, have the trivial bounds
$$\|S_\la^kf\|_{L^\infty}\le C2^{\frac{n-3}2k}\|f\|_{L^{\frac{2n}{n+1}}}.$$
Since 
$$\frac{n-3}{2n}=\frac{(n-1)^2}{2n(n+1)}\cdot \frac{(n+1)(n-3)}{(n-1)^2},$$
$$1-\frac{(n+1)(n-3)}{(n-1)^2}=\frac4{(n-1)^2},$$
and 
$$0=-\frac{2}{n+1}\cdot \frac{(n+1)(n-3)}{(n-1)^2} + \frac4{(n-1)^2}\cdot \frac{n-3}2,$$
if we interpolate between these two estimates, we conclude that
\begin{equation}\label{5.17}
\|S^k_\la f\|_{L^{\frac{2n}{n-3}}}\le C\|f\|_{L^{\frac{2n}{n+1}}},
\end{equation}
for a uniform constant $C$, which is independent of $k=1,2,\dots$.  The pair of exponents in this inequality
corresponds to the point $\alpha$ in Figure~\ref{fig2}, which is one of the endpoints for the range in \eqref{i.1}.

If we use H\"ormander's oscillatory integral theorem \cite{Hos} (see \cite[Theorem 2.1.1]{FIO}) then we can
also obtain the following estimate for the dual exponents satisfying \eqref{i.1},
\begin{equation}\label{5.18}
\|S^k_\la f\|_{L^{\frac{2n}{n-2}}}\le C2^{\frac{n-3}2k}2^{-(n-1)\frac{n-2}{2n}k}\|f\|_{L^{\frac{2n}{n+2}}}
=C2^{-\frac{k}n}\|f\|_{L^{\frac{2n}{n+1}}}
.
\end{equation}
If we interpolate between \eqref{5.17} and \eqref{5.18} we conclude that we have \eqref{5.14} for some 
$\sigma_{r,s}>0$ provided that $r,s$ are as in \eqref{i.1} and $\tfrac{2n}{n+2}\le r<\frac{2n}{n+1}$.  Since
the conditions on the amplitude $a(x,y)$ are symmetric in $x,y$, by duality, we conclude that the same is true
for the dual range of exponents, and, therefore, we have \eqref{5.14} for all $r,s$ as in \eqref{i.1}, which
concludes the proof of Proposition~\ref{prop2.2}.  \qed

\end{document}